\documentclass[a4paper,10pt,reqno]{amsart}
\usepackage[english]{babel}
\usepackage{amsmath,amssymb,amsthm,amscd}
\usepackage{xcolor}
\usepackage{comment}
\usepackage{hyperref}
	\hypersetup{breaklinks=true,colorlinks=true,linkcolor=blue,citecolor=blue}

\usepackage{graphicx}

\theoremstyle{plain}
\newtheorem{theorem}{Theorem}

\newtheorem{lemma}[theorem]{Lemma}

\theoremstyle{definition}

\newtheorem{remark}[theorem]{Remark}

\newcommand{\N}{\ensuremath{\mathbb{N}}}

\newcommand{\R}{\ensuremath{\mathbb{R}}}

\DeclareMathOperator{\e}{e}

\renewcommand{\leq}{\leqslant}
\renewcommand{\geq}{\geqslant}

\begin{document}
\title[Numerical index $\ell_p^2$]{On the numerical index of the real two-dimensional $L_p$ space}
 \thanks{Research partially supported by projects PGC2018-093794-B-I00 (MCIU/AEI/FEDER, UE), P20-00255 (Junta de Andaluc\'{i}a/FEDER, UE), A-FQM-484-UGR18 (Junta de Andaluc\'{i}a/FEDER, UE), and FQM-185 (Junta de Andaluc\'{i}a/FEDER, UE). The second author is also supported by the Ph.D. scholarship FPU18/03057 (MECD)}
 \subjclass[2010]{46B20,\ 47A12}
 \keywords{numerical radius, numerical index, $L_p$-spaces}
 \date{April 27th, 2022}

\maketitle

\centerline{\textsc{\large Javier
		Mer\'{\i}}  \footnote{Corresponding
author. \emph{E-mail:} \texttt{jmeri@ugr.es}} \quad and \quad \textsc{\large Alicia Quero }}

\begin{center} Departamento de An\'{a}lisis Matem\'{a}tico \\ Facultad de
Ciencias \\ Universidad de Granada \\ 18071 Granada, SPAIN \\
\emph{E-mail addresses:} \texttt{jmeri@ugr.es}, \ \texttt{aliciaquero@ugr.es}
\end{center}

\thispagestyle{empty}

\begin{abstract}
We compute the numerical index of the two-dimensional real $L_p$ space for $\frac65\leq p\leq \frac32$ and $3\leq p\leq 6$.
\end{abstract}

\maketitle

\section[\S 1. Introduction]{Introduction}
The numerical index of a Banach space is a constant relating the
norm and the numerical radius of bounded linear operators on
the space. Let us recall the relevant definitions.
Given a Banach space $X$, we will write $X^*$ for its
topological dual and $\mathcal{L}(X)$ for the Banach algebra of
all bounded linear operators on $X$. For an operator $T\in
\mathcal{L}(X)$, its \emph{numerical radius} is defined as
$$
v(T):=\{|x^*(Tx)| \colon x^*\in X^*,\ x\in X,\ \|x^*\|=\|x\|=x^*(x)=1 \}
$$
which is a seminorm on $\mathcal{L}(X)$
satisfying $v(T)\leq \|T\|$ for every $T\in \mathcal{L}(X)$. The \emph{numerical index} of
$X$ is the constant given by
$$
n(X):=\inf\{v(T) \colon T\in \mathcal{L}(X),\ \|T\|=1\}
$$
or, equivalently, $n(X)$ is the greatest constant $k\geq 0$ satisfying
$k\,\|T\| \leq v(T)$ for every $T\in \mathcal{L}(X)$.
Classical references on numerical index are the paper
\cite{D-Mc-P-W} and the monographs by F.F.~Bonsall and J.~Duncan
\cite{B-D1,B-D2} from the seventies. In the last decades this field of study has grown in various directions with the contribution of several authors. The reader will find the
state of the art on the subject in the survey paper
\cite{KaMaPa} and a more recent account in the first chapter of the book \cite{SpearsBook}.

In the following we recall some results concerning the numerical index which
will be relevant to our discussion. It is clear that $0\leq
n(X)\leq 1$ for every Banach space $X$. In the real case, the numerical index can take any value in $[0,1]$. In the complex case, one has
$1/\e\leq n(X)\leq 1$ and all of these values are possible. Let us
also mention that $v(T^*)=v(T)$ for every
$T\in \mathcal{L}(X)$, where $T^*$ is the adjoint operator of $T$ (see
\cite[\S~9]{B-D1}), so it clearly follows that $n(X^*)\leqslant
n(X)$. Although the equality does not
always hold, when $X$ is a reflexive space, one clearly gets
$n(X)=n(X^*)$. There are some classical Banach spaces for
which the numerical index has been calculated. If $H$ is a Hilbert space of dimension
greater than one, then $n(H)=0$ in the real case and $n(H)=1/2$ in
the complex case. Besides, $n(L_1(\mu))=1$ and the same happens to all 
its isometric preduals. In particular, it follows that $n\bigl(C(K)\bigr)=1$
for every compact $K$ and the same is true for all finite-codimensional subspaces of $C[0,1]$. 

The computation of the numerical index of $L_p$-spaces when $p\neq 1,2,\infty$ has proved to be a difficult task and remains as an important open problem in the theory of numerical index since it started. Let us present the known results on the matter. For $1< p<\infty$, 
we write $\ell_p^m$ for the $m$-dimensional
$L_p$-space, $q=p/(p-1)$
for the conjugate exponent of $p$, and
$$
M_p:=\max_{t\in[0,1]} \frac{|t^{p-1}-t|}{1+t^p}=\max_{t\geq 1} \frac{|t^{p-1}-t|}{1+t^p},
$$
which is the numerical radius of the operator represented by the matrix $\begin{pmatrix}0 & 1 \\ -1 & 0 \end{pmatrix}$
defined on the real space $\ell_p^2$. This is stated in \cite[Lemma~2]{MarMer-LP}, where it is also observed that $M_q=M_p$. It is known that the sequence $\bigl(n(\ell_p^m)\bigr)_{m\in\N}$
	is decreasing and that $n\bigl(L_p(\mu)\bigr)=\inf \{n(\ell_p^m)\colon
		m\in\N\}$ for every measure $\mu$ such that
		$\dim\bigl(L_p(\mu)\bigr)=\infty$, this can be found in \cite{Eddari,Eddari2,Eddari3}. Moreover, in the real case, the inequality $n(L_p[0,1])\geq \frac{M_p}{12}$ holds for every $1< p<\infty$ \cite{MMP-Israel}. Also in the real case, one has  $\max\left\{\frac{1}{2^{1/p}},\
	\frac{1}{2^{1/q}}\right\}\,M_p\leq n(\ell_p^{2})\leq
	M_p$ \cite{MarMer-LP}.

Very recently, it has been proved \cite{MQ} that the numerical index is attained at the operator $\begin{pmatrix}0 & 1 \\ -1 & 0 \end{pmatrix}$ for many absolute and symmetric norms on $\R^2$ and, as a major consequence, it is shown that $n(\ell_p^2)=M_p$ for $\frac{3}{2}\leq p\leq 3$. In \cite{MonikaZheng-lp2} the authors polished skilfully the arguments of \cite{MQ} to provide a slight improvement: the equality $n(\ell_p^2)=M_p$ is proved for $1+\alpha_0\leq p \leq \alpha_1$ where $\alpha_0$ is the root of $f(x)=1+x^{-2}-(x^{-\frac{1}{x}}+x^\frac{1}{x})$ and $\frac{1}{1+\alpha_0}+\frac{1}{\alpha_1}=1$ ($\alpha_0\approx0.4547$). The aim of this paper is to show that $n(\ell_p^2)=M_p$ holds for $p\in [\frac65, 6]$. The main difference with previous works is the use of Riesz-Thorin interpolation theorem (see \cite[Theorem~2.1 in chapter 2]{Stein-Shakarchi}, for instance) to estimate the norm of operators on $\ell_p^2$. More precisely, we will use that the inequality 
$$
\|T\|\leq \|T\|_1^{1/p}\|T\|_\infty^{1/q}
$$
holds for every operator $T\in \mathcal{L}(\ell_p^2)$.

To finish the introduction, we recall two facts about
numerical radius that we will need in our discussion. Let $X$ be a Banach space, and suppose that 
$S\in \mathcal{L}(X)$ is an onto isometry. Then, for
every operator $T\in \mathcal{L}(X)$, it is easy to check that
$$
v(T)=v(\pm S^{-1}TS).
$$
The following result, which can be deduced from \cite[Lemma 3.2]{D-Mc-P-W}, will be useful to compute the numerical radius of operators in $\mathcal{L}(\ell_p^2)$.
\begin{lemma}\label{lemma:numerical-radius}
	Let $1<p<\infty$ and $T=\begin{pmatrix}a & b \\ c & d \end{pmatrix}$ be an operator in $\mathcal{L}(\ell_p^2)$. Then
	$$v(T)=\max\left\{\max_{t\in[0,1]}\dfrac{|a+d\,t^p|+|b\,t+c\,t^{p-1}|}{1+t^p},\max_{t\in[0,1]}\dfrac{|d+a\,t^p|+|c\,t+b\,t^{p-1}|}{1+t^p} \right\}.$$
	In particular, $M_p=v\begin{pmatrix} 0 & 1 \\ -1 & 0 \end{pmatrix}=\displaystyle\max_{t\in[0,1]}\dfrac{\left|t^{p-1}-t\right|}{1+t^p}$.
\end{lemma}

\section{The results}
Our first result gives some information about the point $t_0$ where the operator $\begin{pmatrix}0 & 1 \\ -1 & 0 \end{pmatrix}\in\mathcal{L}(\ell_p^2)$ attains its numerical radius which will be of help in the proof of the main theorem.
\begin{lemma}\label{lemma:estimation-t0}
Let $t_0\in]0,1[$ be such that $M_p=\displaystyle\max_{t\in[0,1]}\dfrac{|t^{p-1}-t|}{1+t^p}=\dfrac{|t_0^{p-1}-t_0|}{1+t_0^p}$. 
The inequalities 
$$ 
\left(\dfrac{2p-2}{4-p}\right)^\frac{1}{2-p}\leq t_0\leq \left(\frac{p-1}{2p+1}\right)^\frac{1}{p} \qquad \textnormal{and}\qquad  t_0^{2p-3}\leq\frac{q}{p}\qquad 
$$
hold for every  $p\in [\frac65,\frac32]$.
\end{lemma}
\begin{proof}
	We start showing that $\left(\frac{2p-2}{4-p}\right)^\frac{1}{2-p}\leq t_0\leq \left(\frac{p-1}{2p+1}\right)^\frac{1}{p}$. To do so, take $\xi=\dfrac{p-1}{3}$, define the functions
	$$
	f(t)=\dfrac{t^\xi}{1+t^p} \qquad \textnormal{and} \qquad g(t)=t^{p-1-\xi}-t^{1-\xi} \qquad (t\in[0,1]),
	$$ 	
	and observe that $\frac{t^{p-1}-t}{1+t^p}= f(t)g(t)$ for $t\in[0,1]$.
	It is easy to see that $f$ increases until $t_1=\left(\frac{\xi}{p-\xi}\right)^{\frac{1}{p}}=\left(\frac{p-1}{2p+1}\right)^\frac{1}{p}$ and then decreases since
	$$
	f'(t)= \dfrac{\left(\xi+(\xi-p)t^p\right)}{t^{1-\xi}\left(1+t^p\right)^2}\qquad (0<t<1). 
	$$
	Similarly, $g$ increases until $t_2=\left(\dfrac{p-1-\xi}{1-\xi}\right)^{\frac{1}{2-p}}=\left(\dfrac{2p-2}{4-p}\right)^\frac{1}{2-p}$ and then decreases as
	$$
	g'(t)=\dfrac{(p-1-\xi)-(1-\xi)t^{2-p}}{t^{2-p+\xi}} \qquad (0<t<1).
	$$
	Therefore, the function $\frac{t^{p-1}-t}{1+t^p}=f(t)g(t)$ increases in the interval $]0,\min\{t_1,t_2\}[$ and decreases in the interval $]\max\{t_1,t_2\},1[$, so we deduce that $\min\{t_1,t_2\}\leq t_0\leq \max\{t_1,t_2\}$. Let us show now that $t_2\leq t_1$. To do so, observe that
	
\begin{align*}
\left(\frac{2p-2}{4-p}\right)^\frac{1}{2-p} \leq \left(\frac{p-1}{2p+1}\right)^\frac{1}{p}&\Longleftrightarrow \frac{1}{2-p}\log\left(\frac{2p-2}{4-p}\right)\leq \frac{1}{p}\log\left(\frac{p-1}{2p+1}\right)\\
&\Longleftrightarrow p\left(\log(2)+\log(p-1)-\log\left(4-p\right)\right)\leq(2-p)\big(\log(p-1)-\log(2p+1)\big)\\
&\Longleftrightarrow p\log(2)+(2p-2)\log(p-1)+(2-p)\log(2p+1)\leq p\log\left(4-p\right)
\end{align*}	
	and define the functions $h,k:[\frac65,\frac32]\to \R$ by
	$$
	h(p)=p\log(2)+(2p-2)\log(p-1)+(2-p)\log(2p+1)\qquad \textnormal{and} \qquad k(p)=p\log\left(4-p\right).
	$$
	Since
	\begin{align*}
	k'(p)&=\log(4-p)-\frac{p}{4-p}\\
	k''(p)&=-\frac{1}{4-p}-\frac{4}{(4-p)^2}<0
	\end{align*}
	we get that $k'$ is decreasing which, together with $k'(\frac32)>0$, tells us that $k$ is increasing. Besides, we have that
	\begin{align*}
	h'(p)&=\log(2)+2\log(p-1)+2-\log(2p+1)+\frac{4-2p}{2p+1}\\
	h''(p)&=\frac{2}{p-1}-\frac{2}{2p+1}-\frac{10}{(2p+1)^2}\\
	&=\frac{4p^2+14}{(p-1)(2p+1)^2}>0
	\end{align*}
	and so $h'(p)$ is increasing. This, together with $h'(\frac65)<0$ and $h'(\frac32)>0$, tells us that 
	$$
	\max_{p\in[\frac65,\frac32]}h(p)= \max\left\{h\left(\frac65\right),h\left(\frac32\right)\right\}=h\left(\frac65\right)\,.
	$$
	So $\max_{p\in[\frac65,\frac32]}h(p)=h(\frac65)<k(\frac65)<\min_{p\in[\frac65,\frac32]}k(p)$.
	
	The inequality $t_0^{2p-3}\leq\frac{q}{p}=\frac1{p-1}$ is equivalent to $1\leq \frac{1}{p-1}t_0^{3-2p}$. As
	we already know that $\left(\frac{2p-2}{4-p}\right)^\frac{1}{2-p}\leq t_0 $, the required inequality  will follow if we prove that
		$$ 
		1\leq\dfrac{1}{p-1}\left(\dfrac{2p-2}{4-p}\right)^\frac{3-2p}{2-p}, 
		$$
		which is equivalent to show that
		\begin{equation*}
		 \left(\frac{4-p}{2}\right)^{3-2p}\leq \dfrac{1}{(p-1)^{p-1}}.
		\end{equation*}  
To see this, consider the functions $\phi,\psi:[\frac65,\frac32]\to\R$ given by
$$
\phi(p)=\left(\frac{4-p}{2}\right)^{3-2p} \qquad \textnormal{and} \qquad \psi(p)=\frac{1}{(p-1)^{p-1}}\, , 
$$  
and observe that $\phi$ is decreasing since
	$$
	\phi'(p)= \left(\frac{4-p}{2}\right)^{3-2p} \left(-2\log\left(\frac{4-p}{2}\right)-\frac{3-2p}{4-p}\right)< 0 
	$$
	for $\frac65\leq p< \frac{3}{2}$. Besides, we have that
	$$ 
	\psi'(p)=(p-1)^{(1-p)}\left(-\log(p-1)-1\right)
	$$
	so $\psi$ increases in $]\frac65,1+\frac1\e[$ and decreases in $]1+\frac1\e, \frac32[$. Therefore, we deduce that
	$$
	\min_{p\in[\frac65,\frac32]} \psi(p)=\min\left\{\psi \left(\frac65\right),\psi\left(\frac32\right)\right\}=\psi\left(\frac65\right)>\phi\left(\frac65\right)=\max_{p\in[\frac65,\frac32]} \phi(p)
	$$
	which finishes the proof.
\end{proof}

We are ready to present and prove the main result of the paper.

\begin{theorem}\label{Main-theorem}
Let $p\in \left[\frac65,6\right]$. Then, 
$$
n(\ell_p^2)=M_p=\max_{t\in[0,1]}\dfrac{|t^{p-1}-t|}{1+t^p}.
$$
\end{theorem}

\begin{proof} 
Using that $n(\ell_p^2)=n(\ell_q^2)$ and that the result is already known for $p\in[\frac32,2]$ we only need to work for $1<p\leq \frac32$. We divide the proof into three claims, the first two are valid for all the values of $1<p\leq \frac32$. Observe that 
\begin{align*}
n(\ell_p^2)&=\inf\left\{\dfrac{v(T)}{\|T\|}\colon 0\neq T\in\mathcal{L}(\ell_p^2)\right\} \\
&=\min\left\{ \inf\left\{\dfrac{v(T)}{\|T\|}\colon 0\neq T\in\mathcal{L}(\ell_p^2), \ \|T\|_\infty\leq\|T\|_1 \right\}, \inf\left\{\dfrac{v(T)}{\|T\|}\colon 0\neq T\in\mathcal{L}(\ell_p^2), \ \|T\|_1\leq\|T\|_\infty \right\}\right\}.
\end{align*}  
By \cite[Remark 3]{MarMer-LP}, if $T\in\mathcal{L}(\ell_p^2)$ is such that $\|T\|_\infty\leq\|T\|_1$, then
$M_p=\displaystyle\max_{t\in[0,1]}\dfrac{t^{p-1}-t}{1+t^p}\leq \dfrac{v(T)}{\|T\|}.$ Therefore, it is enough to prove that
\begin{equation}\label{eq:1a-condicion-suficiente}
\inf\left\{\dfrac{v(T)}{\|T\|}\colon 0\neq T\in\mathcal{L}(\ell_p^2), \ \|T\|_1\leq\|T\|_\infty \right\}\geq M_p.
\end{equation}

In order to give a lower estimation for $v(T)$, we make some observations. First, we may suppose that $\|T\|_1=\max\{|a|+|c|,|b|+|d|\}=|a|+|c|$.
Indeed, the operator
$$
S=\begin{pmatrix}0 & 1 \\ 1 & 0 \end{pmatrix}\begin{pmatrix}a & b \\ c & d \end{pmatrix}\begin{pmatrix}0 & 1 \\ 1 & 0 \end{pmatrix}=\begin{pmatrix}d & c \\ b & a \end{pmatrix}
$$
satisfies that $v(S)=v(T)$ and $\|S\|=\|T\|$ since $\begin{pmatrix}0 & 1 \\ 1 & 0 \end{pmatrix}$ is an isometry and $S^{-1}=S$.

In addition, we may assume that $T=\begin{pmatrix}a & b \\ -c & -d \end{pmatrix}$ with $a,b,c,d\geq0$.
Indeed, if $T=\begin{pmatrix}a & b \\ c & d \end{pmatrix}$ with  $a,b,c,d\in\R$, we may consider the operator $S=\begin{pmatrix}|a| & |b| \\ -|c| & -|d| \end{pmatrix}$ which clearly satisfies that $\|S\|=\|T\|$ and $v(S)\leq v(T)$. 

So, from now on we consider operators of the form $T=\begin{pmatrix}a & b \\ -c & -d \end{pmatrix}$ with $a,b,c,d\geq0$ and satisfying $\|T\|_1=a+c\leq \|T\|_\infty$.
For this class of operators, using Lemma~\ref{lemma:numerical-radius}, we have that 
\begin{equation}\label{eq:v(T)-estimation}
\begin{aligned}
v(T)&=\max\left\{\max_{t\in[0,1]}\dfrac{|a-d\,t^p|+|b\,t-c\,t^{p-1}|}{1+t^p},\max_{t\in[0,1]}\dfrac{|d-a\,t^p|+|c\,t-b\,t^{p-1}|}{1+t^p} \right\} \\
& \geq \max\left\{\dfrac{|a-d\,t_0^p|+|b\,t_0-c\,t_0^{p-1}|}{1+t_0^p},\dfrac{|d-a\,t_0^p|+|c\,t_0-b\,t_0^{p-1}|}{1+t_0^p} \right\}
\end{aligned}
\end{equation}
where $t_0$ is taken as in Lemma~\ref{lemma:estimation-t0}. Let us write
$$
F(T)=\dfrac{|a-d\,t_0^p|+|b\,t_0-c\,t_0^{p-1}|}{1+t_0^p} \qquad \textnormal{and} \qquad G(T)=\dfrac{|d-a\,t_0^p|+|c\,t_0-b\,t_0^{p-1}|}{1+t_0^p}
$$
and recall that by Riesz-Thorin theorem we have that $\|T\|\leq\|T\|_1^{1/p}\|T\|_\infty^{1/q}$. Hence, using \eqref{eq:v(T)-estimation}, it is clear that
\begin{align*}
&\inf\left\{\dfrac{v(T)}{\|T\|}\colon 0\neq T\in\mathcal{L}(\ell_p^2), \ \|T\|_1\leq\|T\|_\infty \right\}
\geq \\ &\quad \inf\left\{\dfrac{\max\{F(T),G(T)\}}{\|T\|_1^{1/p}\|T\|_\infty^{1/q}}\colon 0\neq T\in \mathcal{L}(\ell_p^2), T=\begin{pmatrix}a & b \\ -c & -d \end{pmatrix}, \ a,b,c,d\geq0,\ \|T\|_1=a+c\leq\|T\|_\infty \right\}:=(\alpha).
\end{align*}
In view of \eqref{eq:1a-condicion-suficiente}, to prove the theorem it is enough to show that 
\begin{equation*}\label{2a-condicion-suficiente}
(\alpha)\geq \dfrac{t_0^{p-1}-t_0}{1+t_0^p}.
\end{equation*}
To do so, we distinguish three cases: 
\begin{itemize}
\item $\|T\|_1=a+c\leq a+b=\|T\|_\infty$.
\item $\|T\|_1=a+c\leq c+d=\|T\|_\infty$ and  $c+a-d\leq c\,t_0^{2-p}$.
\item $\|T\|_1=a+c\leq c+d=\|T\|_\infty$ and  $c\,t_0^{2-p} \leq c+a-d$.
\end{itemize}

\noindent\textsc{Claim 1.} \emph{Let $1<p\leq\frac{3}{2}$ and let $T=\begin{pmatrix}a & b \\ -c & -d \end{pmatrix}$ be a non-zero operator in $\mathcal{L}(\ell_p^2)$ with $a,b,c,d\geq0$ and $\|T\|_1=a+c\leq a+b=\|T\|_\infty$. Then
$$
(\alpha)\geq \dfrac{t_0^{p-1}-t_0}{1+t_0^p}.
$$}

Observe that it suffices to prove that $(\alpha)_F\geq \frac{t_0^{p-1}-t_0}{1+t_0^p}$ where
$$
(\alpha)_F:=\inf\left\{\dfrac{F(T)}{\|T\|_1^{1/p}\|T\|_\infty^{1/q}}\colon 0\neq T\in \mathcal{L}(\ell_p^2), T=\begin{pmatrix}a & b \\ -c & -d \end{pmatrix}, \ a,b,c,d\geq0,\ \|T\|_1=a+c\leq a+b=\|T\|_\infty \right\}.
$$
Note that the restriction $\|T\|_1=a+c\leq a+b=\|T\|_\infty$ is equivalent to impose $c\leq b$ and $b+d\leq a+c$, which clearly implies $d\leq a$ and so 
	$$
	F(T)=\dfrac{a-d\,t_0^p+|b\,t_0-c\,t_0^{p-1}|}{1+t_0^p}\,.
	$$
	
	In order to estimate $(\alpha)_F$ we may suppose that $b\,t_0^{2-p}\leq c$ (equivalently, $b\,t_0\leq c\,t_0^{p-1}$). Indeed, if otherwise $b\,t_0^{2-p}>c$, we consider the operator $S=\begin{pmatrix}a & b \\ -b\,t_0^{2-p} & -d \end{pmatrix}$ which satisfies the hypotheses of Claim~1 since 
	$$
	\|S\|_1=a+b\,t_0^{2-p}>a+c\geq b+d \qquad \textnormal{and} \qquad \|S\|_\infty=a+b>d+b\,t_0^{2-p}.
	$$
	Moreover, $\|S\|_1>\|T\|_1$, $\|S\|_\infty=\|T\|_\infty$, and $F(S)=\dfrac{a-d\,t_0^p}{1+t_0^p}<F(T)$, so
	$$
	\dfrac{F(S)}{\|S\|_1^{1/p}\|S\|_\infty^{1/q}}<\dfrac{F(T)}{\|T\|_1^{1/p}\|T\|_\infty^{1/q}}.
	$$
	Thus, we can write
	\begin{align*}
	(\alpha)_F&= \inf\left\{\dfrac{F(T)}{\|T\|_1^{1/p}\|T\|_\infty^{1/q}}\colon0\neq T\in \mathcal{L}(\ell_p^2),\ T=\begin{pmatrix}a & b \\ -c & -d \end{pmatrix}, \ a,b,c,d\geq0, \ c\leq b, \ b+d\leq a+c \right\} \\ &\geq\dfrac{1}{1+t_0^p}\inf_{(a,b,c,d)\in A_1} \dfrac{a-d\,t_0^p+c\,t_0^{p-1}-b\,t_0}{(a+c)^{1/p}(a+b)^{1/q}} 
	\end{align*}
	where
	$$A_1=\left\{(a,b,c,d)\in\R^4\backslash\{0\}\colon a,b,c,d\geq0, \ b\,t_0^{2-p}\leq c \leq b , \ b+d \leq a+c \right\}.$$
	From the restriction $d \leq a+c-b$, it follows that 
	$$
	a-d\,t_0^p+c\,t_0^{p-1}-b\,t_0\geq a(1-t_0^p)+b(t_0^p-t_0)+c(t_0^{p-1}-t^p_0)
	$$
	so
	$$
	\inf_{(a,b,c,d)\in A_1} \dfrac{a-d\,t_0^p+c\,t_0^{p-1}-b\,t_0}{(a+c)^{1/p}(a+b)^{1/q}} \geq \inf_{(a,b,c)\in A_2} \dfrac{a(1-t_0^p)+b(t_0^p-t_0)+c(t_0^{p-1}-t^p_0)}{(a+c)^{1/p}(a+b)^{1/q}},
	$$
	where
	$$
	A_2=\left\{(a,b,c)\in\R^3\backslash\{0\}\colon a,b,c\geq0, \ b\,t_0^{2-p}\leq c \leq b\leq a+c\right\}.
	$$
	We define the function $$f(a,b,c)=\dfrac{a(1-t_0^p)+b(t_0^p-t_0)+c(t_0^{p-1}-t^p_0)}{(a+c)^{1/p}(a+b)^{1/q}} \qquad \left((a,b,c)\in\R^3\right)
	$$
	and our goal is to show that
	$$
	\inf_{(a,b,c)\in A_2} f(a,b,c)\geq t_0^{p-1}-t_0.
	$$
	Observe that $f$ decreases in the variable $b$ since $t_0^p-t_0<0$, so using that $b\leq a+c$, it is clear that $f(a,b,c)\geq f(a,a+c,c)$ for every $(a,b,c)\in A_2$. 
	Therefore, we have to minimize the two-variable function given by
	$$
	g(a,c)=f(a,a+c,c)=\dfrac{a(1-t_0)+c(t_0^{p-1}-t_0)}{(a+c)^{1/p}(2a+c)^{1/q}} \qquad \left((a,c)\in\R^2\right)
	$$
	on the set $A_3=\left\{(a,b)\in\R^2\backslash\{0\}\colon a,c\geq0, \ a\,t_0^{2-p}\leq c(1-t_0^{2-p})\right\}$.
	
	To get the inequality $g(a,c)\geq t_0^{p-1}-t_0$ it suffices to show that $g$ is decreasing in $c$, since in such a case it follows that
	$g(a,c)\geq \lim_{c\to\infty}g(a,c)=t_0^{p-1}-t_0$ for every $(a,c)\in A_3$ as desired.
	So let us prove that $\dfrac{\partial g}{\partial c}(a,c)\leq 0$ for every $(a,c)\in A_3$:
	\begin{align*}
	\dfrac{\partial g}{\partial c}(a,c) & \, (a+c)^{1/p+1} (2a+c)^{1/q+1} \\
	& =(t_0^{p-1}-t_0)(a+c)(2a+c) -\left(a(1-t_0)+c(t_0^{p-1}-t_0)\right)\left(\frac{1}{p}(2a+c) + \frac{1}{q}(a+c)\right)
	\\
	&  = \left(2t_0^{p-1}-\frac{1}{q}t_0-1-\frac{1}{p}\right)a^2 + \left(\left(1+\frac{1}{q}\right)t_0^{p-1}-\frac{1}{q} t_0-1\right)ac.
	\end{align*}
	Observe that 
	$$ 
	2t_0^{p-1}-\frac{1}{q}t_0-1-\frac{1}{p}= \left(1+\frac{1}{q}\right)t_0^{p-1}-\frac{1}{q}t_0-1-\frac{1}{p}(1-t_0^{p-1})\leq \left(1+\frac{1}{q}\right)t_0^{p-1}-\frac{1}{q}t_0-1,
	$$
	so to finish the proof of the claim it is enough to show that $\textstyle \left(1+\frac{1}{q}\right)t_0^{p-1}-\frac{1}{q} t_0-1\leq0$. To do so, define the function
	$$
	u(t)=\left(1+\frac{1}{q}\right)t^{p-1}-\frac{1}{q} t-1 \qquad (t\in[0,1])
	$$
	 and note that $u(1)=0$. Thus, the inequality $u(t)\leq 0$ will hold for every $t\in [0,1]$ if we prove that $u$ is an increasing function. This is easy to check as
	$$ 
	u'(t)=\left(1+\frac{1}{q}\right)(p-1)t^{p-2}-\frac{1}{q}
	$$
	and
	$$ 
	u''(t)=\left(1+\frac{1}{q}\right)(p-1)(p-2)t^{p-3}\leq 0
	$$
	for every $t\in]0,1[$, hence $u'$ is decreasing. Since
	$u'(1)=\left(1+\frac{1}{q}\right)(p-1)-\frac{1}{q}=\frac{2}{q}(p-1)\geq0$, it follows that $u'(t)\geq0$ for every $t\in]0,1[$ and $u$ is an increasing function as desired. Therefore, Claim 1 is proved.

Now we consider the operators satisfying $\|T\|_1=a+c\leq c+d=\|T\|_\infty$. Observe that this restriction is equivalent to impose $a\leq d$ and $b+d\leq a+c$ (and also gives $b\leq c$), so from now on we have
$$
F(T)=\dfrac{|a-d\,t_0^p|+c\,t_0^{p-1}-b\,t_0}{1+t_0^p} \qquad \textnormal{and} \qquad G(T)=\dfrac{d-a\,t_0^p+|c\,t_0-b\,t_0^{p-1}|}{1+t_0^p}.
$$

\noindent\textsc{Claim 2.} \emph{Let $1<p\leq\frac{3}{2}$ and let $T=\begin{pmatrix}a & b \\ -c & -d \end{pmatrix}$ be a non-zero operator in $\mathcal{L}(\ell_p^2)$ with $a,b,c,d\geq0$, $\|T\|_1=a+c\leq c+d=\|T\|_\infty$, and $c+a-d\leq c\,t_0^{2-p}$. Then
$$
(\alpha)\geq \dfrac{t_0^{p-1}-t_0}{1+t_0^p}.
$$}

In this case we will use only $G(T)$ to estimate $(\alpha)$. Define 
$$
B_1=\left\{(a,b,c,d)\in\R^4\backslash \{0\} \colon a,b,c,d\geq0, \ a\leq d, \ b\leq c+a-d\leq c\,t_0^{2-p}\right\}
$$ 
and observe that  
	$$
	(\alpha)\geq \dfrac{1}{1+t_0^p}\inf_{(a,b,c,d)\in B_1} \dfrac{d-a\,t_0^p+c\,t_0-b\,t_0^{p-1}}{(a+c)^{1/p}(c+d)^{1/q}}\,.
	$$
	Using that $b\leq c+a-d$, we obtain
	$$
	\inf_{(a,b,c,d)\in B_1} \dfrac{d-a\,t_0^p+c\,t_0-b\,t_0^{p-1}}{(a+c)^{1/p}(c+d)^{1/q}}\geq \inf_{(a,c,d)\in B_2} \dfrac{-a(t_0^{p-1}+t_0^p)-c(t_0^{p-1}-t_0)+d(1+t_0^{p-1})}{(a+c)^{1/p}(c+d)^{1/q}},
	$$
	where $B_2=\left\{(a,c,d)\in\R^3\backslash \{0\} \colon a,c,d\geq0, \ a\leq d-c(1-t_0^{2-p})\right\}$. Therefore, defining
	$$
	f(a,c,d)=\dfrac{-a(t_0^{p-1}+t_0^p)-c(t_0^{p-1}-t_0)+d(1+t_0^{p-1})}{(a+c)^{1/p}(c+d)^{1/q}} \qquad \left((a,c,d)\in\R^3\right),
	$$
	our problem is to show that
	$$
	\inf_{(a,c,d)\in B_2} f(a,c,d)\geq t_0^{p-1}-t_0.
	$$
	It is clear that $f$ is decreasing in $a$, therefore $f(a,c,d)\geq f(d-c(1-t_0^{2-p}),c,d)$ for every $(a,c,d)\in B_2$, so we have to minimize
	$$ g(c,d)=f(d-c(1-t_0^{2-p}),c,d)= \dfrac{c(t_0^p-t_0^2)+d(1-t_0^p)}{(d+c\,t_0^{2-p})^{1/p}(c+d)^{1/q}} \qquad \left((c,d)\in \R^2\right)$$
	on the set $B_3=\left\{(c,d)\in \R^2\backslash\{0\}\colon c,d\geq0, \ c(1-t_0^{2-p})\leq d \right\}$.
	Observe that $g$ is increasing in $d$ since
	\begin{align*}
		\dfrac{\partial g}{\partial d}(c,d) & \, (d+c\,t_0^{2-p})^{1/p+1} (c+d)^{1/q+1} \\
		& = (1-t_0^p)(d+c\,t_0^{2-p})(c+d)-\left(c(t_0^p-t_0^2)+d(1-t_0^p)\right)\left(\frac{1}{p}(c+d)+\frac{1}{q}(d+c\,t_0^{2-p})\right) \\
		&= \frac{1}{q}(1-t_0^p)(c+d)d - \frac{1}{q}(d+c\,t_0^{2-p})\left(c(t_0^p-t_0^2)+d(1-t_0^p)\right) + \left(t_0^{2-p}-\frac{1}{q}t_0^2-\frac{1}{p}t_0^p\right)(c+d)c \\
		&\geq \frac{1}{q}(1-t_0^p)(c+d)d - \frac{1}{q}(c+d)\left(c(t_0^p-t_0^2)+d(1-t_0^p)\right) + \left(t_0^{2-p}-\frac{1}{q}t_0^2-\frac{1}{p}t_0^p\right)(c+d)c \\
		& =(t_0^{2-p}-t_0^p)(c+d)c\geq 0.
	\end{align*}
	Therefore, for every $(c,d)\in B_3$, we have that
	$$g(c,d)\geq g(c,c(1-t_0^{2-p}))=\dfrac{1-t_0^{2-p}}{(2-t_0^{2-p})^{1/q}}=\dfrac{t_0^{p-1}-t_0}{(2t_0^p-t_0^2)^{1/q}}\geq t_0^{p-1}-t_0,$$
	where we have used that 
	$$
	t_0^p\leq\frac{p-1}{2p+1} \leq p-1
	$$
	by Lemma~\ref{lemma:estimation-t0} and so, $2t_0^p-t_0^2\leq 2t_0^p\leq 2p-2\leq 1$.

We consider now the remaining case.

\noindent\textsc{Claim 3.} \emph{Let $\frac{6}{5}\leq p\leq\frac{3}{2}$ and let $T=\begin{pmatrix}a & b \\ -c & -d \end{pmatrix}$ be a non-zero operator in $\mathcal{L}(\ell_p^2)$ with $a,b,c,d\geq0$, $\|T\|_1=a+c\leq c+d=\|T\|_\infty$, and $c\,t_0^{2-p} \leq c+a-d$. Then
$$
(\alpha)\geq \dfrac{t_0^{p-1}-t_0}{1+t_0^p}.
$$}

	First, in order to estimate $(\alpha)$, observe that we may suppose that $a\geq d\,t_0^p$. Indeed, if otherwise $a<d\,t_0^p$, we consider the operator $S=\begin{pmatrix} d\,t_0^p & b \\ -c & -d \end{pmatrix}$ which satisfies the hypotheses of  Claim~3 since  $c\,t_0^{2-p} \leq c+a-d<c+d\,t_0^p-d$,
	$$\|S\|_1=d\,t_0^p+c > a+c\geq b+d, \qquad \textnormal{and} \qquad \|S\|_\infty =c+d\geq b+d>b+d\,t_0^p.$$
	Moreover, $\|S\|_1>\|T\|_1$, $\|S\|_\infty=\|T\|_\infty$,
	$$ 
	F(S)=\dfrac{c\,t_0^{p-1}-b\,t_0}{1+t_0^p}<F(T), \qquad \textnormal{and} \qquad G(S)=\dfrac{d-d\,t_0^{2p}+|b\,t_0^{p-1}-c\,t_0|}{1+t_0^p}<G(T),
	$$
	so 
	$$ 
	\dfrac{\max\{F(T),G(T)\}}{\|T\|_1^{1/p}\|T\|_\infty^{1/q}} \geq \dfrac{\max\{F(S),G(S)\}}{\|S\|_1^{1/p}\|S\|_\infty^{1/q}}. 
	$$
	Additionally, we may assume that $c\,t_0^{2-p}\leq b$. Indeed, if otherwise $c\,t_0^{2-p}>b$, we consider the operator $S=\begin{pmatrix} a & c\,t_0^{2-p} \\ -c & -d \end{pmatrix}$ which satisfies the hypotheses of Claim 3 since $c\,t_0^{2-p} \leq c+a-d$,
	$$
	\|S\|_1=a+c \geq c\,t_0^{2-p}+d,\qquad \textnormal{and} \qquad \|S\|_\infty =c+d > c\,t_0^{2-p}+a.
	$$
	Furthermore, $\|S\|_1=\|T\|_1$, $\|S\|_\infty=\|T\|_\infty$, and, using that $c\,t_0^{2-p}>b$, it is is clear that 
	$$ 
	F(S)=\dfrac{a-d\,t_0^p+c\,t_0^{p-1}-c\,t_0^{2-p}\,t_0}{1+t_0^p}<F(T) \qquad \textnormal{and} \qquad G(S)=\dfrac{d-a\,t_0^{p}}{1+t_0^p}<G(T),
	$$
	so 
	$$ 
	\dfrac{\max\{F(T),G(T)\}}{\|T\|_1^{1/p}\|T\|_\infty^{1/q}} \geq \dfrac{\max\{F(S),G(S)\}}{\|S\|_1^{1/p}\|S\|_\infty^{1/q}}. 
	$$
	Therefore, we assume from now on that $a,b,c,d\geq 0$, $d\,t_0^p\leq a\leq d$, and $c\,t_0^{2-p}\leq b\leq c+a-d$. Under such restrictions, we have
	$$
	F(T)=\dfrac{a-d\,t_0^p+c\,t_0^{p-1}-b\,t_0}{1+t_0^p} \qquad \textnormal{and} \qquad G(T)=\dfrac{d-a\,t_0^p+b\,t_0^{p-1}-c\,t_0}{1+t_0^p}
	$$
	and so our goal is to give a lower bound of 
	$$
	\inf_{(a,b,c,d)\in C_1}\dfrac{\max\{F(T),G(T)\}}{\|T\|_1^{1/p}\|T\|_\infty^{1/q}}
	$$
	where $C_1=\left\{(a,b,c,d)\in\R^4\backslash\{0\}\colon a,b,c,d\geq 0, \ d\,t_0^p\leq a\leq d, \ c\,t_0^{2-p}\leq b\leq c+a-d \right\}$. To do so, notice that
	\begin{equation}\label{eq:claim-3-ineq-F-G}
	F(T)\leq G(T)\Longleftrightarrow a-d\,t_0^p+c\,t_0^{p-1}-b\,t_0 \leq d-a\,t_0^p+b\,t_0^{p-1}-c\,t_0 \Longleftrightarrow b\geq c-(d-a)\dfrac{1+t_0^p}{t_0^{p-1}+t_0},
	\end{equation}
	and the equality holds if and only if $b= c-(d-a)\dfrac{1+t_0^p}{t_0^{p-1}+t_0}$. Our next step is to observe that we may compute the infimum using only operators satisfying $b= c-(d-a)\dfrac{1+t_0^p}{t_0^{p-1}+t_0}$, but first we need to show that
	$$
	ct_0^{2-p}\leq  c-(d-a)\dfrac{1+t_0^p}{t_0^{p-1}+t_0}\leq c+a-d.
	$$
	
	On the   one hand, we claim that ${1+t_0^p}\geq {t_0^{p-1}+t_0}$ and, consequently, $c-(d-a)\dfrac{1+t_0^p}{t_0^{p-1}+t_0}\leq c+a-d$. Consider the function $u\colon[0,1]\to \R$ given by
	$$ 
	u(t)=1+t^p-t^{p-1}-t \qquad (t\in[0,1])
	$$
	and observe that $u(1)=0$. Hence, the inequality ${1+t_0^p}\geq {t_0^{p-1}+t_0}$ will follow immediately if we prove that $u$ decreases in $t$. Indeed,
	$$ u'(t)=p\,t^{p-1}-(p-1)\,t^{p-2}-1$$
	and
	$$ u''(t)=p(p-1)\,t^{p-2}-(p-1)(p-2)\,t^{p-3}\geq 0 $$
	for every $t\in]0,1[$, thus $u'$ is increasing. Since
	$u'(1)=0$, it follows that $u'(t)\leq 0$ for every $t\in]0,1[$ as desired.
	
	On the other hand, we may and do assume that $c\,t_0^{2-p}\leq c-(d-a)\dfrac{1+t_0^p}{t_0^{p-1}+t_0}$. Indeed, if otherwise 
	$c-(d-a)\dfrac{1+t_0^p}{t_0^{p-1}+t_0}<c\,t_0^{2-p}$, we consider the operator $S=\begin{pmatrix} a & c\,t_0^{2-p} \\ -c & -d \end{pmatrix}$ which satisfies the conditions in $C_1$. Moreover, $\|S\|_1=\|T\|_1$, $\|S\|_\infty=\|T\|_\infty$, and $G(S)=\dfrac{d-a\,t_0^{p}}{1+t_0^p}<G(T)$. Now, it follows from \eqref{eq:claim-3-ineq-F-G} that $F(T)< G(T)$ and $F(S)< G(S)$ as $c-(d-a)\dfrac{1+t_0^p}{t_0^{p-1}+t_0}<c\,t_0^{2-p}\leq b$, therefore
	$$\max\{F(S),G(S)\}=G(S)<G(T)=\max\{F(T),G(T)\}.$$
	
	Consequently,
	\begin{align*}
	\inf_{(a,b,c,d)\in C_1} \dfrac{\max\{F(T),G(T)\}}{\|T\|_1^{1/p}\|T\|_\infty^{1/q}} &\geq \inf_{(a,b,c,d)\in C_2} \dfrac{\max\{F(T),G(T)\}}{\|T\|_1^{1/p}\|T\|_\infty^{1/q}}
	\end{align*}
	where 
	$$
	C_2=\left\{(a,b,c,d)\in \R^4\backslash \{0\} \colon a,b,c,d\geq 0, \ d\,t_0^p\leq a \leq d, \ c\,t_0^{2-p}\leq b\leq c+a-d, \ c\,t_0^{2-p}\leq c-(d-a)\textstyle{\frac{1+t_0^p}{t_0^{p-1}+t_0}} \right\}.
	$$
	Let us observe that the infimum 
	$$
	\inf_{(a,b,c,d)\in C_2} \dfrac{\max\{F(T),G(T)\}}{\|T\|_1^{1/p}\|T\|_\infty^{1/q}}
	$$ 
	can be computed using only operators satisfying $F(T)=G(T)$, that is, satisfying $b=c-(d-a)\dfrac{1+t_0^p}{t_0^{p-1}+t_0}$. Indeed, if $T$ is an operator satisfying the conditions in $C_2$, consider the operator 
	$$
	S=\begin{pmatrix} a & c-(d-a)\dfrac{1+t_0^p}{t_0^{p-1}+t_0} \\ -c & -d \end{pmatrix}
	$$ 
	which clearly satisfies the conditions in $C_2$, $F(S)=G(S)$, and $\|S\|_1^{1/p}\|S\|_\infty^{1/q}=\|T\|_1^{1/p}\|T\|_\infty^{1/q}$. Observe that if $b\leq c-(d-a)\dfrac{1+t_0^p}{t_0^{p-1}+t_0}$ then 
	$\max\{F(T), G(T)\}=F(T)\geq F(S)$. If otherwise $b\geq c-(d-a)\dfrac{1+t_0^p}{t_0^{p-1}+t_0}$ then 
		$\max\{F(T), G(T)\}=G(T)\geq G(S)$. So in either case we have $F(S)=G(S)\leq \max\{F(T), G(T)\}$.  
	
	For operators satisfying satisfying $b=c-(d-a)\dfrac{1+t_0^p}{t_0^{p-1}+t_0}$ we have that
	\begin{align*}
	F(T)=G(T)&=\dfrac{a\left(1-t_0\dfrac{1+t_0^p}{t_0^{p-1}+t_0}\right)+d\left(t_0\dfrac{1+t_0^p}{t_0^{p-1}+t_0}-t_0^p\right)+c(t_0^{p-1}-t_0)}{1+t_0^p}\\
	&=\dfrac{a\,\dfrac{t_0^{p-1}-t_0^{p+1}}{t_0^{p-1}+t_0}+d\,\dfrac{t_0-t_0^{2p-1}}{t_0^{p-1}+t_0}+c(t_0^{p-1}-t_0)}{1+t_0^p}\,.
	\end{align*}
	Hence, defining the function
	$$
	f(a,c,d)=\frac{a\,\dfrac{t_0^{p-1}-t_0^{p+1}}{t_0^{p-1}+t_0}+d\,\dfrac{t_0-t_0^{2p-1}}{t_0^{p-1}+t_0}+c(t_0^{p-1}-t_0)}{(a+c)^{1/p}(c+d)^{1/q}} \qquad \left((a,c,d)\in\R^3\right)
	$$
	we have that
	$$
	\inf_{(a,b,c,d)\in C_2} \dfrac{\max\{F(T),G(T)\}}{\|T\|_1^{1/p}\|T\|_\infty^{1/q}}\geq \dfrac{1}{1+t_0^p} \inf_{(a,c,d)\in C_3} f(a,c,d)
	$$
	where 
	$$
	C_3=\left\{(a,c,d)\in \R^3\backslash \{0\} \colon a,c,d\geq 0, \ d\,t_0^p\leq a \leq d, \ c\,t_0^{2-p}\leq c-(d-a)\textstyle{\frac{1+t_0^p}{t_0^{p-1}+t_0}} \right\}.
	$$
	Our aim is to prove that $f$ decreases in $c$ for $\frac65\leq p\leq\frac{3}{2}$. In such a case, it is clear that 
	$$
	f(a,c,d)\geq \lim_{c\to\infty}f(a,c,d)=t_0^{p-1}-t_0
	$$ 
	for every $(a,c,d)\in C_3$ and, as a consequence, $(\alpha)\geq \dfrac{t_0^{p-1}-t_0}{1+t_0^p}$ as desired.

	So, let us show that $\dfrac{\partial f}{\partial c}(a,c,d)\leq 0$ for every $(a,c,d)\in C_3$. Calling $K=a\,\dfrac{t_0^{p-1}-t_0^{p+1}}{t_0^{p-1}+t_0}+d\,\dfrac{t_0-t_0^{2p-1}}{t_0^{p-1}+t_0}$, we can write
	\begin{align*}
	\dfrac{\partial f}{\partial c}(a,c,d) & \, (a+c)^{1/p+1} (c+d)^{1/q+1} \\
	& = (t_0^{p-1}-t_0)(a+c)(c+d)-\left(c(t_0^{p-1}-t_0)+K\right)\left(\frac{1}{p}(c+d)+\frac{1}{q}(a+c)\right) \\
	& = \left((t_0^{p-1}-t_0)\left(\frac{a}{p}+\frac{d}{q}\right)-K\right) c + (t_0^{p-1}-t_0)ad-K\left(\frac{a}{q}+\frac{d}{p}\right)
	\\
	&\leq \left((t_0^{p-1}-t_0)\left(\frac{a}{p}+\frac{d}{q}\right)-K\right)\left(c+\frac{a}{q}+\frac{d}{p}\right),
	\end{align*}
	where we have used that
	$$
	\left(\frac{a}{p}+\frac{d}{q}\right)\left(\frac{a}{q}+\frac{d}{p}\right) = \frac{a^2+d^2}{pq}+\left(\frac{1}{p^2}+\frac{1}{q^2}\right)ad=\frac{(a-d)^2}{pq}+ad\geq ad. 
	$$
	Now, observe that
	\begin{align*}
	(t_0^{p-1}-t_0)\left(\frac{a}{p}+\frac{d}{q}\right)-K\leq 0 &\Longleftrightarrow (t_0^{2p-2}-t^2_0)\left(\frac{a}{p}+\frac{d}{q}\right)\leq a(t_0^{p-1}-t_0^{p+1})+d(t_0-t_0^{2p-1})\\
	&\Longleftrightarrow a\left(t_0^{p-1}-t_0^{p+1}-\frac{1}{p}t_0^{2p-2}+\frac{1}{p}t_0^2 \right) \geq d\left(\frac{1}{q}t_0^{2p-2}-\frac{1}{q}t_0^2-t_0+t_0^{2p-1}\right). 
	\end{align*}
	To prove the last inequality, using that $ a\geq d\,t_0^p$, it is enough to prove that 
	$$
	t_0^{2p-1}-t_0^{2p+1}-\frac{1}{p}t_0^{3p-2}+\frac{1}{p}t_0^{p+2}  \geq \frac{1}{q}t_0^{2p-2}-\frac{1}{q}t_0^2-t_0+t_0^{2p-1} 
	$$
	which, decomposing $t_0=\frac1pt_0+\frac1q t_0$ and $t_0^{2p+1}=\frac1pt_0^{2p+1}+\frac1q t_0^{2p+1}$, is equivalent to
	\begin{align*} \frac{1}{p}\left(t_0+t_0^{p+2}-t_0^{2p+1}-t_0^{3p-2}\right) &\geq \frac{1}{q} \left(t_0^{2p-2}+t_0^{2p+1}-t_0^2-t_0\right)
	\\ &= \frac{1}{q} t_0^{2p-3}\left(t_0+t_0^4-t_0^{5-2p}-t_0^{4-2p}\right).
	\end{align*}
	Note that $t_0+t_0^{p+2}-t_0^{2p+1}-t_0^{3p-2} \geq t_0+t_0^4-t_0^{5-2p}-t_0^{4-2p}$ for $\frac{6}{5}\leq p\leq\frac{3}{2}$ as 
	$$ t_0^{p+2}-t_0^4+t_0^{5-2p}-t_0^{2p+1}+t_0^{4-2p}-t_0^{3p-2}= t_0^{p+2}\left(1-t_0^{2-p}\right)+t_0^{5-2p}\left(1-t_0^{4p-4}\right)+t_0^{4-2p}\left(1-t_0^{5p-6}\right)\geq 0,$$
	therefore, it suffices to show that $t_0^{2p-3}\leq\frac{q}{p}$\,. 
	But this inequality holds for $p\in[\frac65,\frac32]$ thanks to Lemma~\ref{lemma:estimation-t0} and so Claim 3 is proved.
\end{proof}

\begin{remark}
The only use of the restriction $\frac65\leq p$ in the above proof was to guarantee that the inequality 
	$$ 
	t_0^{p+2}-t_0^4+t_0^{5-2p}-t_0^{2p+1}+t_0^{4-2p}-t_0^{3p-2}= t_0^{p+2}\left(1-t_0^{2-p}\right)+t_0^{5-2p}\left(1-t_0^{4p-4}\right)+t_0^{4-2p}\left(1-t_0^{5p-6}\right)\geq 0
	$$
holds. This inequality remains true for some values of $p$ smaller than $\frac65$ but close to it. The same happens with Lemma~\ref{lemma:estimation-t0}, so our procedure can give the equality $n(\ell_p^2)=M_p$ for a little wider range of values of $p$. However it seems that it does not work for $p$ close to $1$. Indeed, for $p=1.16$, numerical computations give $t_0\approx 0.073924$ and $M_p\approx 0.558064$. Besides, the operator $T=\begin{pmatrix}a & b \\ -c & -d \end{pmatrix}$ with
$$
a=0.0487295, \qquad b=13.639181, \qquad c=15, \qquad \textnormal{and} \qquad d=1
$$
satisfies
$$
\frac{\max\{F(T),G(T)\}}{\|T\|_1^{1/p}\|T\|_\infty^{1/q}}=\frac{1}{1+t_0^p}\frac{a-d\,t_0^p+c\,t_0^{p-1}-b\,t_0}{(a+c)^{1/p}(c+d)^{1/q}}\approx0.557895<M_p.
$$
\end{remark}

\end{document}